\documentclass{imsart}

\def\bfm#1{\mbox{\boldmath$#1$}}
\RequirePackage{mathtools}
\RequirePackage{amsthm,amsmath,amsfonts,amssymb}
\RequirePackage{enumerate}
\RequirePackage[authoryear]{natbib}
\RequirePackage[colorlinks,citecolor=blue,urlcolor=blue]{hyperref}
\RequirePackage{graphicx}

\startlocaldefs
\theoremstyle{plain}

\newtheorem{theorem}{Theorem}[section]
\newtheorem{lemma}[theorem]{Lemma}
\newtheorem{proposition}[theorem]{Proposition}
\theoremstyle{remark}

\newtheorem*{remark}{Remark}

\endlocaldefs

\begin{document}
\renewcommand{\theequation}{\thesection.\arabic {equation}}


\newcommand{\bbB}{{\mathbb{B}}}
\newcommand{\bbC}{{\mathbb{C}}}
\newcommand{\bbD}{{\mathbb{D}}}
\newcommand{\bbE}{{\mathbb{E}}}
\newcommand{\bbI}{{\mathbb{I}}}
\newcommand{\bbJ}{{\mathbb{J}}}
\newcommand{\bbK}{{\mathbb{K}}}
\newcommand{\bbN}{{\mathbb{N}}}
\newcommand{\bbO}{{\mathbb{O}}}
\newcommand{\bbR}{{\mathbb{R}}}
\newcommand{\bbS}{{\mathbb{S}}}
\newcommand{\bbT}{{\mathbb{T}}}
\newcommand{\bbV}{{\mathbb{V}}}
\newcommand{\bbX}{{\mathbb{X}}}
\newcommand{\bbY}{{\mathbb{Y}}}


\newcommand{\cA}{{\cal A}}         \newcommand{\dA}[1]{{\cal A}_{#1}}
\newcommand{\cB}{{\cal B}}         \newcommand{\dB}[1]{{\cal B}_{#1}}
\newcommand{\cI}{{\cal I}}
\newcommand{\cJ}{{\cal J}}
\newcommand{\cK}{{\cal K}}
\newcommand{\cN}{{\cal N}}
\newcommand{\cO}{{\cal O}}
\newcommand{\cS}{{\cal S}}         \newcommand{\dS}[1]{{\cal S}_{#1}}
\newcommand{\cR}{{\cal R}}
\newcommand{\cT}{{\cal T}}
\newcommand{\cU}{{\cal U}}
\newcommand{\cV}{{\cal V}}
\newcommand{\cX}{{\cal X}}
\newcommand{\cY}{{\cal Y}}
\newcommand{\cZ}{{\cal Z}}


\newcommand{\f}[1]{f_{_{\tiny\rm #1}}}  \newcommand{\hf}[1]{\hat{f}_{#1}}
\newcommand{\tf}[2]{f_{{#1}}^{{#2}}}

\newcommand{\h}[1]{h_{_{\tiny\rm #1}}}
\newcommand{\n}[1]{n_{_{\tiny\rm #1}}}
\newcommand{\q}[2]{q_{_{\tiny\rm #1\!, \, #2}}}
\newcommand{\del}[1]{\delta_{_{\tiny\rm #1}}}
\newcommand{\pii}[2]{\pi_{_{\tiny\rm #1\!, \, #2}}}

\newcommand{\y}[2]{y_{#1}^{\rm #2}}
\newcommand{\yb}[1]{{\bar{y}^{\rm #1}}}
\newcommand{\Y}[1]{Y_{\rm #1}}
\newcommand{\YY}[1]{Y_{\! \mbox{\scriptsize #1}}}

\newcommand{\no}[2]{\|#1\|_{_{#2}}}
\newcommand{\nor}[3]{\|#1\|_{_{#2}}^{#3}}

\newcommand{\bismall}[2]{(\!\!\begin{array}{c} #1 \\ \vspace*{-0.6cm} \\ #2 \end{array} \!\! )}
\newcommand{\bi}[2]{\Big(\!\!\begin{array}{c} #1 \\ #2 \end{array} \!\!\Big)}
\newcommand{\twoc}[2]{\bigg(\!\!\begin{array}{c} #1 \\ #2  \end{array} \!\!\bigg)}
\newcommand{\twolr}[2]{\left(\!\!\begin{array}{c} #1 \\ #2  \end{array} \!\!\right)}

\newcommand{\twob}[2]{\lbrace {#1 \atop #2} \rbrace}
\newcommand{\twoB}[2]{\bigg\{\!\!\begin{array}{c} #1 \\ #2  \end{array} \!\!\bigg\}}

\newcommand{\fourc}[4]{\bigg(\!\!\begin{array}{cc} #1 & #2 \\ #3 & #4 \end{array} \!\!\bigg)}
\newcommand{\fourlr}[4]{\left(\!\!\begin{array}{cc} #1 & #2 \\ #3 & #4 \end{array} \!\!\right)}

\newcommand{\threec}[3]{\left(\!\!\begin{array}{c} #1 \\ #2 \\ #3  \end{array} \!\!\right)}
\newcommand{\threetwo}[6]{\left(\!\!\begin{array}{cc}
#1 & #2 \\
#3 & #4 \\
#5 & #6
\end{array} \!\!\right)}
\newcommand{\threethree}[9]{\left(\!\!\begin{array}{ccc}
#1 & #2 & #3 \\
#4 & #5 & #6 \\
#7 & #8 & #9
\end{array} \!\!\right)}

\newcommand{\sr}[1]{\stackrel{#1}{=}}
\newcommand{\mc}[1]{\multicolumn{1}{c}{#1}}
\newcommand{\MC}[3]{\multicolumn{#1}{#2}{#3}}
\newcommand{\TC}[4]{\contentsline {#1}{\numberline {#2}#3}{#4}}
\newcommand{\LIST}[3]{\contentsline {section}{\numberline {\hspace*{-0.5cm}#1}#2}{#3}}

\newcommand{\Hpi}[1]{\hat{\pi}_{_{\tiny\rm #1}}}
\newcommand{\titem}[1]{\vspace*{-0.15cm}\item[{\rm #1} ]\hspace*{0.1cm}}
\newcommand{\ub}[1]{\underline{\bf #1}}
\newcommand{\RED}[1]{{\color{red} #1}}

\newcommand{\two}[2]{{#1}_{_{\tiny\rm #2}}}
\newcommand{\three}[3]{{#1}_{_{#2, \tiny\rm #3}}}

\newcommand{\threev}[3]{\left(\!\!\begin{array}{c} #1 \\   #2 \\   #3 \end{array}
                          \!\!\right)}


\newcommand{\0}{{\bf 0}\!\!\!{\bf 0}}
\newcommand{\1}{{\bf 1}\!\!\!{\bf 1}}        \def\b1{{\bf 1\!\!\!1}}  


\newcommand{\ba}{{\bf a}}
\newcommand{\bb}{{\bf b}}
\newcommand{\bd}{{\bf d}}
\newcommand{\bm}{{\bf m}}
\newcommand{\bn}{{\bf n}}
\newcommand{\bp}{{\bf p}}
\newcommand{\bq}{{\bf q}}
\newcommand{\br}{{\bf r}}
\newcommand{\bs}{{\bf s}}
\newcommand{\bt}{{\bf t}}
\newcommand{\bu}{{\bf u}}
\newcommand{\bv}{{\bf v}}
\newcommand{\bw}{{\bf w}}
\newcommand{\bx}{{\bf x}}
\newcommand{\by}{{\bf y}}
\newcommand{\bz}{{\bf z}}

\newcommand{\bA}{{\bf A}}
\newcommand{\bB}{{\bf B}}
\newcommand{\bC}{{\bf C}}
\newcommand{\bE}{{\bf E}}
\newcommand{\bD}{{\bf D}}
\newcommand{\bH}{{\bf H}}
\newcommand{\bI}{{\bf I}}
\newcommand{\bJ}{{\bf J}}
\newcommand{\bL}{{\bf L}}
\newcommand{\bM}{{\bf M}}
\newcommand{\bN}{{\bf N}}
\newcommand{\bO}{{\bf O}}
\newcommand{\bP}{{\bf P}}
\newcommand{\bQ}{{\bf Q}}
\newcommand{\bR}{{\bf R}}
\newcommand{\bS}{{\bf S}}
\newcommand{\bT}{{\bf T}}
\newcommand{\bU}{{\bf U}}
\newcommand{\bV}{{\bf V}}
\newcommand{\bW}{{\bf W}}
\newcommand{\bX}{{\bf X}}
\newcommand{\bY}{{\bf Y}}
\newcommand{\bZ}{{\bf Z}}


\newcommand{\ibb}{{\bfm b}}
\newcommand{\ibe}{{\bfm e}}
\newcommand{\ibp}{{\bfm p}}
\newcommand{\ibt}{{\bfm t}}
\newcommand{\ibv}{{\bfm v}}
\newcommand{\ibx}{{\bfm x}}
\newcommand{\iby}{{\bfm y}}
\newcommand{\ibz}{{\bfm z}}
\newcommand{\ibu}{{\bfm u}}
\newcommand{\ibw}{{\bfm w}}
\newcommand{\ibs}{{\bfm s}}
\newcommand{\ibA}{{\bfm A}}
\newcommand{\ibB}{{\bfm B}}
\newcommand{\ibC}{{\bfm C}}
\newcommand{\ibR}{{\bfm R}}
\newcommand{\ibU}{{\bfm U}}
\newcommand{\ibV}{{\bfm V}}
\newcommand{\ibW}{{\bfm W}}
\newcommand{\ibX}{{\bfm X}}


\newcommand{\hb}{\hat{b}}
\newcommand{\hh}{\hat{h}}
\newcommand{\hm}{\hat{m}}
\newcommand{\hp}{\hat{p}}
\newcommand{\hq}{\hat{q}}
\newcommand{\hR}{\hat{R}}

\newcommand{\hbr}{\hat{\br}}

\newcommand{\hth}{\hat{\theta}}   \newcommand{\hTh}{\hat{\Theta}}
\newcommand{\hbth}{{\boldsymbol{\hth}}}
\newcommand{\hbbe}{{\boldsymbol{\hbe}}}
\newcommand{\hbp}{{\boldsymbol{\hp}}}
\newcommand{\heta}{\hat{\eta}}
\newcommand{\hvth}{\hat{\vartheta}}
\newcommand{\hvphi}{\hat{\varphi}}
\newcommand{\hmu}{\hat{\mu}}
\newcommand{\hbmu}{{\boldsymbol{\hmu}}}
\newcommand{\hbSi}{{\boldsymbol{\hSi}}}
\newcommand{\hsi}{\hat{\sigma}}   \newcommand{\hSi}{\hat{\Sigma}}
\newcommand{\hal}{\hat{\alpha}}   \newcommand{\hbe}{\hat{\beta}}
\newcommand{\hga}{\hat{\gamma}}
\newcommand{\hpsi}{\hat{\psi}}
\newcommand{\hxi}{\hat{\xi}}
\newcommand{\hpi}{\hat{\pi}}
\newcommand{\hla}{\hat{\lambda}}
\newcommand{\hbla}{{\boldsymbol{\hla}}}

\newcommand{\whVar}{\widehat{\mbox{Var}}}
\newcommand{\whse}{\widehat{\mbox{se}}}
\newcommand{\whPr}{\widehat{\Pr}}
\newcommand{\whCorr}{\widehat{\mbox{Corr}}}

\newcommand{\tD}{\tilde{D}}
\newcommand{\tx}{\tilde{x}}
\newcommand{\ty}{\tilde{y}}

\newcommand{\tth}{\tilde{\theta}}     \newcommand{\tbth}{\tilde{{\bfm \theta}}}
\newcommand{\tpi}{\tilde{\pi}}
\newcommand{\tmu}{\tilde{\mu}}
\newcommand{\tbmu}{\tilde{\bmu}}
\newcommand{\tbe}{\tilde{\beta}}
\newcommand{\tsi}{\tilde{\sigma}}     \newcommand{\tSi}{\tilde{\Sigma}}
 \newcommand{\tbSi}{\tilde{\bSi}}
\newcommand{\tpsi}{\tilde{\psi}}
\newcommand{\txi}{\tilde{\xi}}
\newcommand{\tbp}{\tilde{{\bfm p}}}
\newcommand{\tp}{\tilde{p}}

\newcommand{\Bu}{\bar{u}}
\newcommand{\Bv}{\bar{v}}
\newcommand{\Bw}{\bar{w}}
\newcommand{\Bx}{\bar{x}}
\newcommand{\By}{\bar{y}}       \newcommand{\BY}{\bar{Y}}
\newcommand{\Bz}{\bar{z}}

\newcommand{\Bxi}{\bar{\xi}}

\newcommand{\BVar}{\overline{\mbox{Var}}}


\newcommand{\al}{\alpha}
\newcommand{\be}{\beta}
\newcommand{\ga}{\gamma}            \newcommand{\Ga}{\Gamma}
\newcommand{\de}{\delta}            \newcommand{\De}{\Delta}
\newcommand{\la}{\lambda}           \newcommand{\La}{\Lambda}
\newcommand{\Th}{\Theta}
\newcommand{\thx}{\theta_x}         \newcommand{\Thx}{\Theta_x}
\newcommand{\thy}{\theta_y}

\newcommand{\si}{\sigma}            \newcommand{\Si}{\Sigma}
\newcommand{\ka}{\kappa}
\newcommand{\om}{\omega}            \newcommand{\Om}{\Omega}

\newcommand{\ve}{\varepsilon}
\newcommand{\vp}{\varphi}
\newcommand{\vr}{\varrho}
\newcommand{\vth}{\vartheta}


\newcommand{\bxi}{{\bfm \xi}}      \newcommand{\bet}{{\bfm \eta}}
\newcommand{\bphi}{{\bfm \phi}}    \newcommand{\bep}{{\bfm \epsilon}}
\newcommand{\bmu}{{\bfm \mu}}      \newcommand{\bnu}{{\bfm \nu}}
\newcommand{\bla}{{\bfm \lambda}}  \newcommand{\bLa}{{\bfm \Lambda}}
\newcommand{\bSi}{{\bfm \Sigma}}
\newcommand{\bom}{{\bfm \omega}}   \newcommand{\bOm}{{\bfm \Omega}}
\newcommand{\bde}{{\bfm \delta}}   \newcommand{\bDe}{{\bfm \Delta}}
\newcommand{\bth}{\boldsymbol{\theta}}
\newcommand{\bsi}{{\bfm \sigma}}
\newcommand{\bbe}{{\bfm \beta}}
\newcommand{\bpsi}{{\bfm \psi}}
\newcommand{\bpi}{{\bfm \pi}}


\newcommand{\Bernoulli}{\mbox{Bernoulli}}
\newcommand{\Binomial}{\mbox{Binomial}}
\newcommand{\BBinomial}{\mbox{BBinomial}}
\newcommand{\BP}{\mbox{BP}}
\newcommand{\Categorical}{\mbox{Categorical}}
\newcommand{\D}{\mbox{D}}
\newcommand{\Degenerate}{\mbox{Degenerate}}
\newcommand{\DExponential}{\mbox{DExponential}}
\newcommand{\Dirichlet}{\mbox{Dirichlet}}
\newcommand{\DMultinomial}{\mbox{DMultinomial}}
\newcommand{\Exponential}{\mbox{Exponential}}
\newcommand{\Exp}{\mbox{Exp}}
\newcommand{\FDiscrete}{\mbox{FDiscrete}}
\newcommand{\GD}{\mbox{GD}}
\newcommand{\GDirichlet}{\mbox{GDirichlet}}
\newcommand{\GLiouville}{\mbox{GLiouville}}
\newcommand{\GP}{\mbox{GP}}
\newcommand{\GPD}{\mbox{GPD}}
\newcommand{\GPZIP}{\mbox{GP\_ZIP}}
\newcommand{\GPoisson}{\mbox{GPoisson}}
\newcommand{\Hgeometric}{\mbox{Hgeometric}}
\newcommand{\HPP}{\mbox{HPP}}
\newcommand{\IBeta}{\mbox{IBeta}}
\newcommand{\Ichi}{\mbox{I}\raisebox{0.5ex}{$\chi$}}
\newcommand{\IGamma}{\mbox{IGamma}}
\newcommand{\IGaussian}{\mbox{IGaussian}}
\newcommand{\IWishart}{\mbox{IWishart}}
\newcommand{\Laplace}{\mbox{Laplace}}
\newcommand{\Liouville}{\mbox{Liouville}}
\newcommand{\Logistic}{\mbox{Logistic}}
\newcommand{\Lognormal}{\mbox{Lognormal}}
\newcommand{\mBeta}{\mbox{Beta}}
\newcommand{\mGamma}{\mbox{Gamma}}
\newcommand{\Multinomial}{\mbox{Multinomial}}
\newcommand{\Normal}{\mbox{Normal}}
\newcommand{\NBinomial}{\mbox{NBinomial}}
\newcommand{\ND}{\mbox{ND}}
\newcommand{\NDirichlet}{\mbox{NDirichlet}}
\newcommand{\NHPP}{\mbox{NHPP}}
\newcommand{\NIW}{\mbox{NIW}}
\newcommand{\Poisson}{\mbox{Poisson}}
\newcommand{\TBeta}{\mbox{TBeta}}
\newcommand{\TN}{\mbox{TN}}
\newcommand{\Wishart}{\mbox{Wishart}}
\newcommand{\ZIGP}{\mbox{ZIGP}}
\newcommand{\ZIP}{\mbox{ZIP}}
\newcommand{\ZOIP}{\mbox{ZOIP}}
\newcommand{\ZINB}{\mbox{ZINB}}


\newcommand{\Corr}{\mbox{Corr}}
\newcommand{\Cov}{\mbox{Cov}}
\newcommand{\se}{\mbox{se}}
\newcommand{\Se}{\mbox{Se}}
\newcommand{\SE}{\mbox{SE}}
\newcommand{\tr}{\mbox{$\,$tr$\,$}}
\newcommand{\rank}{\mbox{rank}\,}
\newcommand{\Var}{\mbox{Var}}
\newcommand{\MSE}{\mbox{MSE}}
\newcommand{\CV}{\mbox{CV}}
\newcommand{\median}{\mbox{median}}

\newcommand{\logit}{\mbox{logit}}
\newcommand{\diag}{\mbox{diag}}
\newcommand{\data}{\mbox{data}}
\newcommand{\KL}{\mbox{KL}}
\newcommand{\IG}{\mbox{IG}}
\newcommand{\I}{\mbox{I}}

\newcommand{\qand}{\quad \mbox{and} \quad}
\newcommand{\qas}{\quad \mbox{as} \quad}
\newcommand{\qag}{\quad \mbox{against} \quad}
\newcommand{\qor}{\quad \mbox{or} \quad}
\newcommand{\qve}{\quad \mbox{versus} \quad}
\newcommand{\col}{\mbox{: }}
\newcommand{\RE}{\mbox{RE}}
\newcommand{\yes}{\mbox{yes}}
\newcommand{\No}{\mbox{no}}
\newcommand{\DPP}{\mbox{DPP}}

\newcommand{\e}{\mbox{e}}
\newcommand{\w}{\mbox{w}}
\renewcommand{\ge}{\geqslant}
\renewcommand{\le}{\leqslant}
\newcommand{\rd}{\,\mbox{d}}


\newcommand{\II}{{\rm I\!I}}
\newcommand{\III}{I$\!$I$\!$I}
\newcommand{\IR}{{I\!\! R}}
\newcommand{\IV}{I$\!$V}
\newcommand{\et}{{\it et al}.}
\newcommand{\Et}{{\it et al}.\,}
\newcommand{\st}{{\it s}.{\it t.}}

\newcommand{\yikH}{y_{ik}^{\rm H}}
\newcommand{\ybH}{\bar{y}^{\rm H}}


\newcommand{\sd}{\stackrel{d}{=}}
\newcommand{\sdt}{$ {\small $\sd$} $}
\newcommand{\heq}{\,\hat{=}\,}
\newcommand{\iid}{\stackrel{{\rm iid}}{\sim}}
\newcommand{\tiid}{i.i.d.$\hspace*{0.08cm}$}
\newcommand{\ind}{\stackrel{{\rm ind}}{\sim}}
\newcommand{\dsim}{\stackrel{.}{\sim}}
\newcommand{\dis}{\displaystyle}
\newcommand{\tex}{\textstyle}
\newcommand{\cf}{cf.$\hspace*{0.1cm}$}

\newcommand{\T}{\!\top\!}
\newcommand{\na}{\nabla}
\newcommand{\noi}{\noindent}
\newcommand{\ra}{\rightarrow}
\newcommand{\pr}{\propto}
\newcommand{\eq}{\equiv}
\newcommand{\pa}{\partial}
\newcommand{\ol}{\overline}
\newcommand{\non}{\nonumber}
\newcommand{\ap}{\approx}
\newcommand{\Bot}{\;\bot\;}
\newcommand{\inde}{{\Bot\!\!\!\!\!\!\!\Bot}}
\newcommand{\btu}{\bigtriangleup}


\newcommand{\vs}{\vspace*{-0.25cm}}
\newcommand{\vkl}{\vskip 0.10in}
\newcommand{\vkL}{\vskip 0.15in}
\newcommand{\vkU}{\vskip 0.30in}


\newcommand{\namelistlabel}[1]{\mbox{#1}\hfil}
\newenvironment{namelist}[1]{%
\begin{list}{}
       {\let \makelabel \namelistlabel
        \settowidth{\labelwidth}{#1}
        \setlength{\leftmargin}{1.1\labelwidth}   }
        }{%
\end{list} }

\def\bds{\begin{description} \itemsep=-\parsep \itemindent=-0.7 cm}
\def\eds{\end{description}}
\def\i{\item}

\newcommand{\hphi}{\hat{\phi}}
\newcommand{\tphi}{\tilde{\phi}}
\newcommand{\tla}{\tilde{\lambda}}
\newcommand{\bepsilon}{\bfm{\epsilon}}
\newcommand{\tnu}{\tilde{\nu}}

\begin{frontmatter}
\title{On singular values of large dimensional  lag-$\tau$ sample auto-correlation matrices}
\runtitle{LSD of auto-correlation matrix}

\begin{aug}
\author[A]{\fnms{Zhanting} \snm{Long}\ead[label=e1]{12032886@mail.sustech.edu.cn}},
\author[A]{\fnms{Zeng} \snm{Li}\ead[label=e2,mark]{liz9@sustech.edu.cn}}
\and
\author[B]{\fnms{Ruitao} \snm{Lin}\ead[label=e3]{rlin@mdanderson.org}}
\address[A]{Southern University of Science and Technology, \printead{e1,e2}}

\address[B]{The University of Texas MD Anderson Cancer Center, \printead{e3}}
\end{aug}

\begin{abstract}
We study the limiting behavior of singular values of a lag-$\tau$ sample auto-correlation matrix $\bR_{\tau}^{\bep}$ of error term $\bep$ in the high-dimensional factor model. We establish the limiting spectral distribution (LSD) which characterizes the global spectrum of $\bR_{\tau}^{\bep}$, and derive the limit of its largest singular value. All the asymptotic results are derived under the high-dimensional asymptotic regime where the data dimension and sample size go to infinity proportionally. Under mild assumptions, we show that 
the LSD of $\bR_{\tau}^{\bep}$ is the same as that of the lag-$\tau$ sample auto-covariance matrix. Based on this asymptotic equivalence, we additionally show that the largest singular value of $\bR_{\tau}^{\bep}$ converges almost surely to the right end point of the support of its LSD. Our results take the first step to identify the number of factors in factor analysis using lag-$\tau$ sample auto-correlation matrices. 
Our theoretical results are fully supported by numerical experiments as well. 
\end{abstract}



\end{frontmatter}

\section{Introduction}
Consider a sequence of $p$-dimensional stationary random vectors $\left\lbrace\by_{i}\right\rbrace$ that has a factor structure and can be represented as  
\begin{equation*}
\by_{i} = \bmu + {\bf B}{\bf f}_{i}+\bep_{i}, \quad i = 1,\ldots,n
\end{equation*}
where $\left\lbrace{\bf f}_{i}\right\rbrace$ is a sequence of $k$-dimensional latent factor vector, and $\left\lbrace\bep_{i}\right\rbrace$ is a sequence of unobservable stochastic error vector of independent and identically distributed (i.i.d.) components with zero mean and unit variance, independent with $\left\lbrace{\bf f}_{i}\right\rbrace$. Determining the number of factors $k$ is a core problem for the factor model, and it possesses many challenges in the high-dimensional setting. \citet{bai2002determining} first proposed a consistent estimator for static factor models. \citet{hallin2007determining} developed an information criterion for dynamic factor models. \citet{lam2012factor} studied the factor model for high-dimensional time series based on  lagged auto-covariance matrices. \citet{fan2020estimating} proposed an estimator based on sample correlation matrices to overcome the issue of the heterogeneous scales of the observed variables. In this paper, we study the lagged sample auto-correlation matrix for two reasons. On one hand, we believe that compared with the sample covariance matrix alone, the auto-correlation matrices of different lags may contain more information of $k$. Our ultimate goal is to investigate whether or not 
by borrowing information from the auto-correlation matrices of different lags, the final inference on the unknown number of factors would be more accurate or efficient. On the other hand, as with \citet{fan2020estimating}, the lag-$\tau$ auto-correlation matrix overcomes the disadvantage of the heterogeneity among different components by self normalization. 

Mathematically,  given the sequence of random vectors $\left\lbrace\by_{i}\right\rbrace$, we denote the population covariance matrix, the lag-$\tau$ (with $\tau$ being a fixed positive integer) auto-covariance, and auto-correlation matrices of $\left\lbrace\by_{i}\right\rbrace$ as $\bSi_{0}^{\by} = {\rm cov}(\by_{i})$, $\bSi_{\tau}^{\by} = {\rm cov}(\by_{i},\by_{i+\tau})$ and $\bOm_{\tau}^{\by} = {\rm corr}(\by_{i},\by_{i+\tau})$, respectively. Similarly, 
the population auto-covariance or auto-correlation matrices can be defined for sequences $\left\lbrace \bep_{i}\right\rbrace$ and $\left\lbrace{\bf f}_{i}\right\rbrace$ by way of analogy.
For example, $\bSi^{{\bf f}}_{\tau} = {\rm cov}({\bf f}_{i},{\bf f}_{i+\tau})$ is the lag-$\tau$ auto-covariance of $\left\lbrace{\bf f}_{i}\right\rbrace$.  Let the superscript ``$t$'' denote the transpose of a vector or matrix. It is known that the lag-$\tau$  auto-correlation matrix
\begin{equation*}  
\bOm^{\by}_{\tau} = [{\rm diag}(\bSi^{\by}_{0})]^{-1/2}\bSi^{\by}_{\tau}[{\rm diag}(\bSi^{\by}_{0})]^{-1/2} = [{\rm diag}(\bSi^{\by}_{0})]^{-1/2}(\bB\bSi^{{\bf f}}_{\tau}\bB^{t})[{\rm diag}(\bSi^{\by}_{0})]^{-1/2},    
\end{equation*}
 exactly has $k$ non-null singular values. As a result, based on the i.i.d observed data sample $\by_{1},\ldots,\by_{n}$, the number of factors $k$ can be  naturally estimated via the singular values of sample version of the lag-$\tau$ auto-correlation matrix 
\begin{equation*}
\bR_{\tau}^{\by} = [{\rm diag}(\bS_{0}^{\by})]^{-1/2}\bS_{\tau}^{\by}[{\rm diag}(\bS_{0}^{\by})]^{-1/2}.
\end{equation*}
Note that, the  lag-$\tau$ sample  auto-covariance matrix is given by 
\begin{eqnarray*}
\bS_{\tau}^{\by} &=& \frac{1}{n-1}\sum\limits_{i=1}^{n}(\by_{i}-\bar{\by})(\by_{i+\tau}-\bar{\by})^{t} \\
&=& \bB\left(\frac{1}{n-1}\sum\limits_{i = 1}^{n}({\bf f}_{i}-\bar{\bf f})({\bf f}_{i+\tau}-\bar{\bf f})^{t}\right)\bB^{t} + \bB\left(\frac{1}{n-1}\sum\limits_{i = 1}^{n}({\bf f}_{i}-\bar{\bf f})(\bep_{i+\tau}-\bar{\bep})^{t}\right) \\
&&+ \left(\frac{1}{n-1}\sum\limits_{i = 1}^{n}(\bep_{i}-\bar{\bep})({\bf f}_{i+\tau}-\bar{\bf f})^{t}\right)\bB^{t} + \frac{1}{n-1}\sum\limits_{i = 1}^{n}(\bep_{i}-\bar{\bep})(\bep_{i+\tau}-\bar{\bep})^{t}\\
&\triangleq & \bP^{\bB}_{\tau}+\bS_{\tau}^{\bep},
\end{eqnarray*}
where for a sequence $\left\lbrace \ba_{i}\right\rbrace=\left\lbrace \by_{i}\right\rbrace, \left\lbrace \bep_{i}\right\rbrace,$ or $\left\lbrace {\bf f}_{i}\right\rbrace$, $\bar{\ba} = \sum_{i=1}^{n}\ba_{i}/n$ and by convention $\ba_{i} = \ba_{n+i}$ for $i = 1,\ldots,\tau$.
Since $\bP^{\bB}_{\tau}$ is of rank $k$,  the lag-$\tau$ sample auto-covariance matrix of $\left\lbrace \by_{i}\right\rbrace$, $\bS_{\tau}^{\by}$,  can be treated as a finite rank perturbation of the lag-$\tau$ sample auto-covariance matrix of $\left\lbrace \bep_{i}\right\rbrace$, $\bS_{\tau}^{\bep}$, which is of rank $p\gg k$. Consequently, the lag-$\tau$  sample auto-correlation matrix of  $\left\lbrace \by_{i}\right\rbrace$, $\bR_{\tau}^{\by}$, is also a finite rank perturbation of the lag-$\tau$ sample auto-correlation matrix of $\left\lbrace \bep_{i}\right\rbrace$, $\bR_{\tau}^{\bep}$,  where
\begin{equation*}
\bR_{\tau}^{\bep} = [{\rm diag}(\bS_{0}^{\bep})]^{-1/2}\bS_{\tau}^{\bep}[{\rm diag}(\bS_{0}^{\bep})]^{-1/2},
\end{equation*}
Hence $\bR_{\tau}^{\by}$ follows the spike model pattern which is well studied in the random matrix theory (RMT), see, \citet{johnstone2001distribution}, \citet{baik2006eigenvalues}, \citet{bai2008central}, \citet{benaych2011eigenvalues}. 

In order to estimate $k$, a clear picture is needed for the asymptotic behavior of the singular values of $\bR_{\tau}^{\by}$, which are effected by the finite rank matrix and $\bR_{\tau}^{\bep}$. As a result, studying the sample auto-correlation matrix of $\left\lbrace \bep_{i}\right\rbrace$, $\bR_{\tau}^{\bep}$, takes the first step to identify the number of factors in factor analysis. In this paper, we study the limiting singular value distribution and the limit of the largest singular value of $\bR_{\tau}^{\bep}$ under the high-dimensional setting where the dimension $p$ and sample size $n$ are assumed to be of the same order. 



Because the eigenvalues of certain large random matrices play a critical role in many multivariate statistical analyses, 
limiting spectral properties of various matrix models has been widely studied using the RMT. In this paper, we use the tools of RMT to study the limiting spectral properties of the lag-$\tau$ sample auto-correlation matrix.
There is a rich literature on the LSD and extreme eigenvalues of large-dimensional matrices. As  a pioneering work, \citet{10.2307/1970079,10.2307/1970008} discovered the LSD of a large dimensional Wigner matrix and the limiting distribution is known as the semicircle law. 
\citet{marvcenko1967distribution} found that the empirical spectral distribution of sample covariance matrix converges to the Mar{\v{c}}enko-Pastur law under mild conditions. Considering the product of random matrices, \citet{YIN1983489}, and \citet{YIN198650} investigated the LSD of $\bS_{n}\bA$, where $\bS_{n}$ is sample covariance matrix and $\bA$ is a positive definite matrix. \citet{bai2007limit} exhibited the existence of LSD of $\bS_{n}\bH$ where $\bH$ is an arbitrary Hermitian matrix, and also investigated the LSD of $\bS_{n}\bW$ where $\bW$ is a Wigner matrix. \citet{yin1983limiting}, \citet{bai1986limiting} showed the existence of the LSD of multivariate $F$-matrix. \citet{bai1988limiting}, \citet{10.1214/aos/1176345134} and \citet{silverstein1985limiting} derived the explicit form of the LSD of multivariate $F$-matrix. The form of $\bH+\bX\bD\bX^{t}$,  where $\bH$ is a Hermitian matrix, $\bD$ is diagonal, and $\bX$ contains independent columns, has been studied by \citet{silverstein1995empirical}. \citet{bose2002limiting} derived the LSD of a circulant matrix. The limiting distributions of eigenvalues of sample correlation matrices was discovered by \citet{jiang2004limiting}. For a high-dimensional time series structure, \citet{li2015singular} investigated the limiting singular value distribution of sample auto-covariance matrices. Most results are derived via the tools of Stieltjes transform and moment method.

As for the limiting behaviour of extreme eigenvalues, the first known result was established by \citet{geman1980limit}, who showed that the largest eigenvalue of a sample covariance matrix convergences to a limit almost surely under a growth condition on all the moments. \citet{yin1988limit} improved this result under the existence of the fourth moment. For Wigner matrix, \citet{bai1988necessary} found the sufficient and necessary conditions for the almost sure convergence of the largest eigenvalue. \citet{jiang2004limiting} showed the largest eigenvalue of a sample correlation almost surely convergences to the right edge of support of its LSD.  \citet{vu2007spectral} derived the upper bound for the spectral norm of symmetric random matrices with independent entries.  \citet{wang2016moment} established the convergence of the largest singular value of a sample auto-covariance matrix based on graph theory. 

The results derived in this paper heavily rely on the pioneer work of  \citet{jiang2004limiting} and \citep{li2015singular}. In particular, 
 \citet{jiang2004limiting} showed that the LSD of the sample correlation matrix $\bR_{0}^{\bep}$ is the same as that of the sample covariance matrix $\bS_{0}^{\bep}$ and also established the convergence of the largest eigenvalue of $\bR_{0}^{\bep}$. Indeed, inspired by \citet{jiang2004limiting}, we try to relate the asymptotic results of singular values of $\bR_{\tau}^{\bep}$ to $\bS_{\tau}^{\bep}$ for fixed  $\tau \geq 1$. Since $\bR_{\tau}^{\bep}$ is not symmetric, we equivalently investigate the limiting behavior of eigenvalues of $\bR_{\tau}^{*} = \bR_{\tau}^{\bep}(\bR_{\tau}^{\bep})^{t}$. We show that the LSD of $\bR_{\tau}^{*}$ is the same as the LSD of $\bS_{\tau}^{*} = \bS_{\tau}^{\bep}(\bS_{\tau}^{\bep})^{t}$ \citep{li2015singular}, mimicking the case of  $\bR_{0}^{\bep}$ and $\bS_{0}^{\bep}$  as shown in \citet{jiang2004limiting}.  
 Additionally, we also prove that the largest eigenvalue of $\bR_{\tau}^{*}$ converges almost surely to the right edge of support of its LSD.

The rest of the paper is organized as follows. Section 2 introduces the main theoretical results in this paper. The detailed proofs of the theorems are given in section 3.

\section{Main results}
\label{sec:result}
\subsection{Preliminary}
Let $\mu$ be a finite measure on the real line, the Stieltjes transform of $\mu$ is defined by 
\begin{equation*}
m_{\mu}\left(z\right) = \int \frac{1}{x-z}\mu\left(dx\right), z\in \mathbb{C}\setminus {\Gamma}_{\mu},
\end{equation*}
where $\Gamma_{\mu}$ is the support of the finite measure $\mu$ on the real line $\mathbb{R}$. 

Let $\bA_{n}$ be a $p\times p$  Hermitian matrix with eigenvalues $\lambda_{1},\lambda_{2},\ldots,\lambda_{p}$, the empirical spectral distribution (ESD) of $\bA_{n}$ is  
\begin{equation*}
F^{\bA_{n}}\left(x\right) = \frac{1}{p}\sum\limits_{j = 1}^{p} I\left\lbrace\lambda_{j}\leq x\right\rbrace, x\in \mathbb{R}.
\end{equation*}
The LSD is the limiting distribution of $\left\lbrace F^{\bA_{n}}\right\rbrace_{n\geq 1}$ for a sequence of random matrices $\left\lbrace\bA_{n}\right\rbrace_{n\geq 1}$. By the definition of $F^{\bA_{n}}$, the Stieltjes transform of the ESD $F^{\bA_{n}}$ is 
\begin{equation*}
m_{\bA_{n}}\left(z\right) = \int \frac{1}{x-z}F^{\bA_{n}}\left(dx\right) = \frac{1}{p}{\rm tr}\left(\bA_{n}-z \bI_{p}\right)^{-1},
\end{equation*}
where ${\rm tr}(\cdot)$ denotes the trace function and $\bI_p$ is the $p$-dimensional identity matrix. 
With $m_{\bA_{n}}\left(z\right)$, the density function of the LSD of $\bA_{n}$ can be obtained by inversion formula,
\begin{equation*}
f\left(u\right) = \lim\limits_{\epsilon \rightarrow 0_{+}} {\mathfrak{Im}} m_{\bA_{n}}\left(u+i\epsilon\right),
\end{equation*}
where $z$ is substituted by $u+i\epsilon$.

\subsection{Limiting spectral distribution}
Recall that $\by_{i} = \bmu+\bB{\bf f}_{i} +\bep_{i}$, $i = 1,\ldots,n$, we first focus on the limiting singular value distribution of the lag-$\tau$ auto-correlation matrix ${\bf R}_{\tau}^{\bep}$. Equivalently, we consider the LSD of the symmetric matrix $\bR_{\tau}^{*} = \bR_{\tau}^{\bep}(\bR_{\tau}^{\bep})^{t}$.

\begin{itemize} 
	\item {\bf Assumption (A)}. $\bep_{i} = \left(\epsilon_{i,1},\ldots,\epsilon_{i,p}\right)^{t}$, $i = 1,2,\ldots,n$ are independent $p$-dimensional random vectors with entries satisfying
	\begin{eqnarray*}
		E\left(\epsilon_{i,j}\right) = 0, \quad E\left(\epsilon_{i,j}^{2}\right) = 1,\quad \sup\limits_{1\leq i\leq n, 1\leq j\leq p} E\left(\left\vert \epsilon_{i,j}^{4+\delta}\right\vert\right) < M,
	\end{eqnarray*}
	for constant $M$ and positive $\delta$.
	
	\item {\bf Assumption (B)}. As $p \rightarrow \infty$, $n \rightarrow \infty$ and $p/n \rightarrow y \in \left(0,\infty\right)$.
\end{itemize}

\begin{theorem}\label{Theorem1}

Under Assumptions (A) and (B), as $p,n \rightarrow \infty$, for fixed $\tau \geq 1$, almost surely the empirical distribution of $F^{\bR^{*}_{\tau}}$ converges to a deterministic probability function $F$ whose Stieltjes transform  $m = m\left(z\right)$, $z \in \mathbb{C}\setminus \mathbb{R}$, satisfies the following equation
\begin{equation*}
z^{2}y^{2}m^{3}+zy\left(y-1\right)m^{2}-zm-1=0.
\end{equation*} 
The density function of $F$, $f\left(u\right)$, is given by
\begin{align}\label{den}
f\left(u\right) =& \frac{1}{y\pi u}\Bigg\{ -u-\frac{5\left(y-1\right)^{2}}{3}+\frac{2^{4/3}\left(3u+\left(y-1\right)^{2}\right)\left(y-1\right)}{3d\left(u\right)^{1/3}}+\frac{2^{2/3}\left(y-1\right)d\left(u\right)^{1/3}}{3} \notag \\
&+\frac{1}{48}\left[-8\left(y-1\right)+ \frac{2^{4/3}\left(3u+\left(y-1\right)^{2}\right)}{d\left(u\right)^{1/3}}+2^{2/3}d\left(u\right)^{1/3}\right]^{2} \Bigg\}^{1/2},
\end{align}
where 
\begin{equation*}
d(u) = -2\left(y-1\right)^{3}+9\left(1+2y\right)u+3\sqrt{3}\sqrt{u\left(-4u^{2}+\left(-1+4y\left(5+2y\right)\right)u-4y\left(y-1\right)^{3}\right)}.
\end{equation*}
Here, the support of $f\left(u\right)$ is $\left(0,b\right]$ for $0<y<1$, and $\left[a,b\right]$ for $y\geq 1$, where 
\begin{eqnarray}
a &=&  \frac{1}{8}\left(-1+20y+8y^{2}-\left(1+8y\right)^{3/2}\right),\label{a} \\
b &=&  \frac{1}{8}\left(-1+20y+8y^{2}+\left(1+8y\right)^{3/2}\right). \label{b}
\end{eqnarray}
For the latter case with $y\geq 1$, the density function $f\left(u\right)$ has an additional point mass $\left(1-\frac{1}{y}\right)$ at the origin.
\end{theorem}

\autoref{fig:my_label2} contrasts the ESD of $\bR_{\tau}^{*}$ (histogram) with $\tau = 1$ and the theoretical limiting density function $f(u)$ (solid line) based on i.i.d. samples from the standard normal distribution with $y = 0.5,1,2,2.5$ and $n = 500$. It can be seen that the empirical histogram of eigenvalues of $\bR_{\tau}^{*}$ is consistent with the limiting  density function \eqref{den} for all $(p,n)$ combinations. 

\begin{figure} 
    \centering
    \includegraphics[width = 14.5cm]{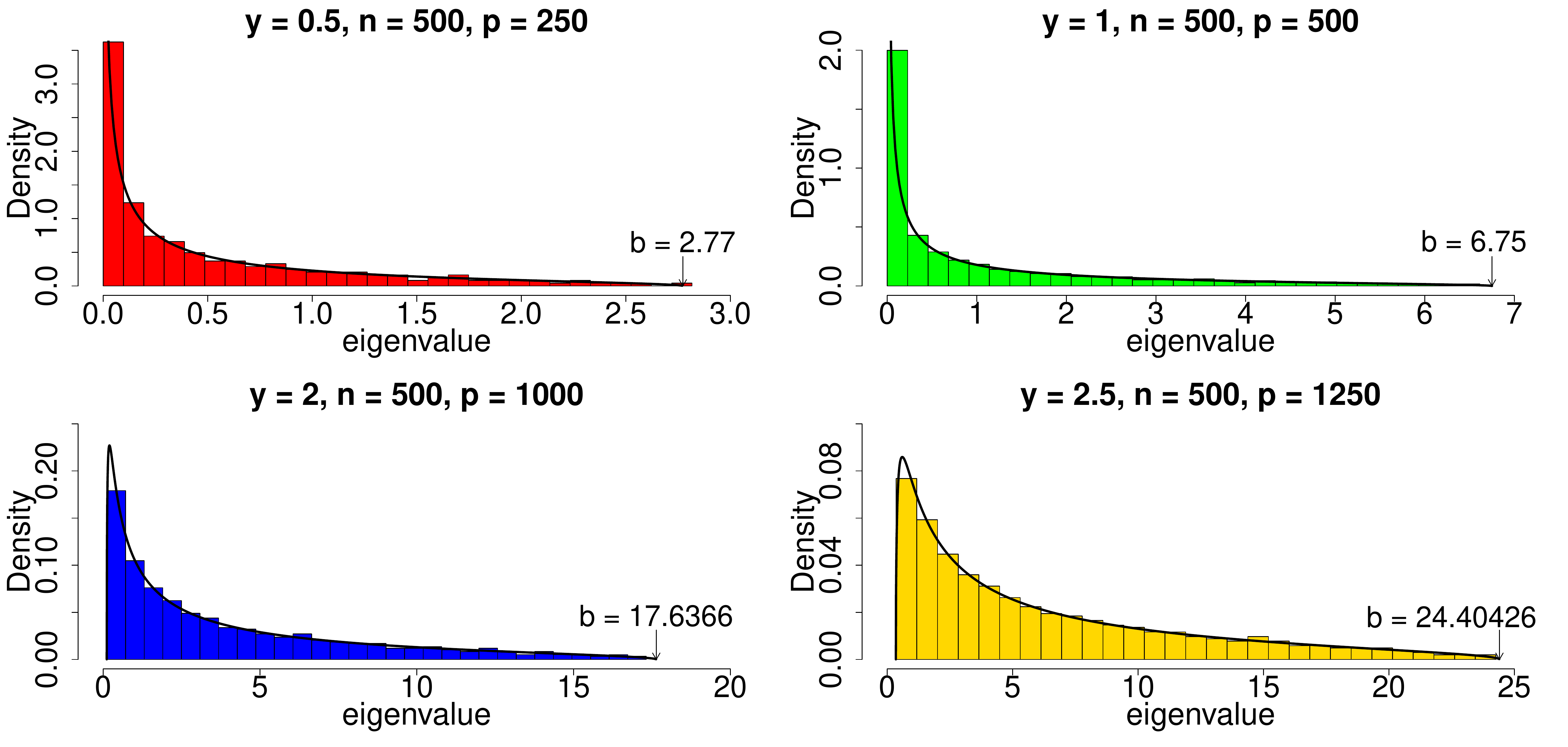}
    \caption{The histogram of the sample eigenvalues of $\bR_{\tau}^{*}$ with $\tau = 1$ and the theoretical limiting spectral density function $f(u)$. In all panels, the sample size $n$ is fixed at $n=500$, and the ratio of the dimensionality to the sample size  is set as $y =0.5,1,1.5,2$ from top to bottom and left to right, respectively.}
    \label{fig:my_label2}
\end{figure}

\begin{remark}
	By comparing \autoref{Theorem1} with Theorem 2.1 in \cite{li2015singular}, we can see that the LSDs of $\bR^{*}_{1}$ and $\bS^{*}_{1}$ are the same, which is consistent with the results on  sample correlation and covariance matrices \citep{jiang2004limiting}. In addition, as shown by \cite{li2015singular} the singular value distribution of $\bS_{\tau}^{\bep}$ is the same as that of $\bS_{1}^{\bep}$ for any fixed $\tau>1$. Such results also hold for the singular value distribution of $\bR_{\tau}^{\bep}$.
\end{remark}

\subsection{Limiting behaviour of the largest eigenvalue}
Next, we study the limiting behaviour of the largest eigenvalue of $\bR_{\tau}^{*}$. The following theorem shows that the largest eigenvalue converges to the right edge of the support of LSD of $\bR_{\tau}^{*}$, mimicking the limiting behavior of the largest eigenvalue of $\bS_{\tau}^{*}$.
\begin{theorem}\label{Theorem2}
Suppose that Assumptions (A) and (B) hold. Let $\lambda_{\rm max}({\bR_{\tau}^{*}})$ be the largest eigenvalue of $\bR_{\tau}^{*}$,  then for fixed $\tau \geq 1$ almost surely,
\begin{equation*}
\lambda_{\rm max}({\bR_{\tau}^{*}}) \rightarrow b,\quad \mbox{ as } p,n\rightarrow \infty,
\end{equation*}
where $b = \frac{1}{8}\left(-1+20y+8y^{2}+\left(1+8y\right)^{3/2}\right)$ is the right edge of the support of the LSD of  $\bR_{\tau}^{*}$.
\end{theorem}

\begin{remark}
The limit of the largest eigenvalue of $\bR^{*}_{\tau}$ is equal to that of $\bS^{*}_{\tau}$.
\end{remark}

\begin{figure}[htbp]
    \centering
    \includegraphics[width = 14.5cm]{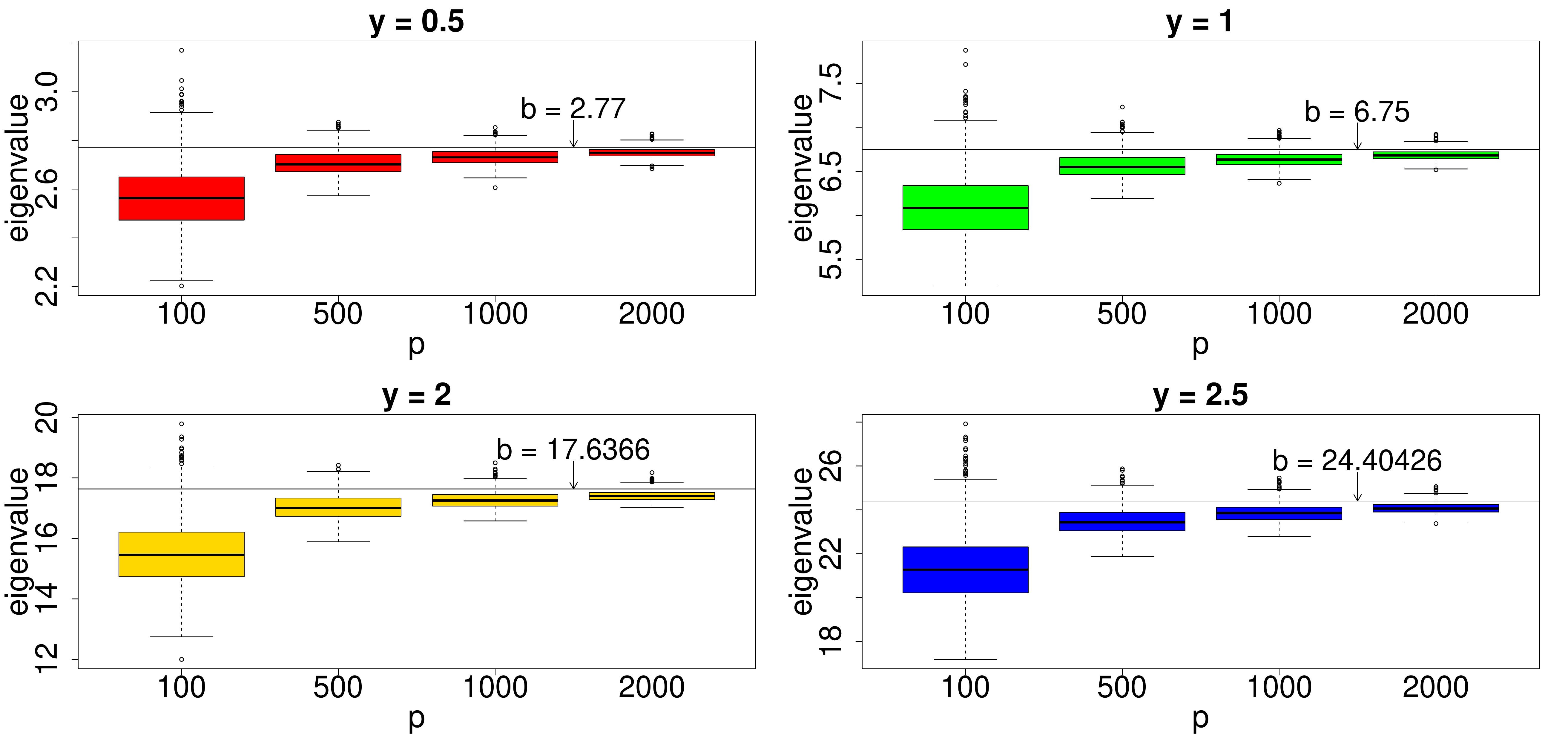}
    \caption{ Boxplot of the largest eigenvalues of 1000 lag-$1$ sample auto-correlation matrices of samples from the standard normal population. In all panels, the horizontal line indicates the right end point of its LSD, and the ratio of the dimensionality to the sample size  is set as $y =0.5,1,2,2.5$ from top to bottom and left to right, respectively.}
    \label{fig:my_label33}
\end{figure}
\autoref{fig:my_label33} displays the boxplot of the largest eigenvalues of $\bR_{\tau}^{*}$ with $\tau=1$ based on 1,000 replications of independent and identically distributed samples from the standard normal distribution. We consider four values for the dimension, i.e., $p = 100,500,1000,2000$, and vary the value of $y$, i.e., the ratio of the dimensionality to the sample size, from 0.5 to 2.5 in the four panels. In each panel, the horizontal line corresponds to the  theoretical right end point $b$ of                                                                                                                                                                                                                                                                                                                                                                                                                                                                                                                                                                                                                                                                                                                                                                                                                                                                                                                                                                                                                                                                                                                                                                                                                                                                                                                                                                                                                                                                                                                                                                                                                                                                                                 LSD. From \autoref{fig:my_label33}, we can see that the largest eigenvalue of $\bR_{\tau}^{*}$ converges to the  right end point $b$ as both the dimension $p$ and the sample size $n$ increase proportionally.

\subsection{Comparison with sample correlation matrix}
In the previous sections, we study the lag-$\tau$ sample auto-correlation matrix for fixed  $\tau\geq 1$. These asymptotic results can not be directly extended to the case of $\bR_{0}^{*}$. Because the LSD of $\bR_{0}^{*}$ is no longer the same as in \autoref{Theorem1}. Unlike $\bR_{\tau}^{\bep}$ for fixed $\tau\geq 1$, $\bR_{0}^{\bep}$ is a symmetric matrix. The limiting behavior can be directly derived based on  the sample correlation matri $\bR_{0}^{\bep}$  and there is no need to consider the eigenvalues of the transformation $\bR_{0}^{*} =\bR_{0}^{\bep}(\bR_{0}^{\bep})^t$. Although \citet{jiang2004limiting} has already showed that the ESD of   $\bR_{0}^{\bep}$ converges to the well-known Mar{\v{c}}enko-Pastur law, for completeness, we copy the results of $\bR_{0}^{\bep}$  below. 
\begin{proposition}{(\cite{jiang2004limiting})}
Suppose $\bep_{i} = \left(\epsilon_{1i},\ldots,\epsilon_{pi}\right)^{t}$, $i = 1,2,\ldots,n$ are independent $p$-dimensional random vectors with entries satisfying $E(\epsilon_{ij})=0$,  $E\left(\vert\epsilon_{ji}\vert^{2}\right)< \infty$. Let $p/n\rightarrow y\in (0,\infty)$, then, almost surely, $F^{\bR_{0}^{\bep}}$ converges to a deterministic probability distribution with density function
\begin{equation*}
f_{y}\left(u\right)= 
\begin{cases}
\frac{1}{2\pi uy}\sqrt{\left(b-u\right)\left(u-a\right)}, & \mbox{if } a\leq u \leq b, \\
0, & \mbox{otherwise,} 
\end{cases}
\end{equation*}
and a point mass with value $1-1/y$ at $x = 0$ if $y>1$,where $a =\left(1-\sqrt{y}\right)^{2}$ and $b =\left(1+\sqrt{y}\right)^{2}$.
\end{proposition}

\begin{figure}[!htbp]
    \centering
    \includegraphics[width = 14.5cm]{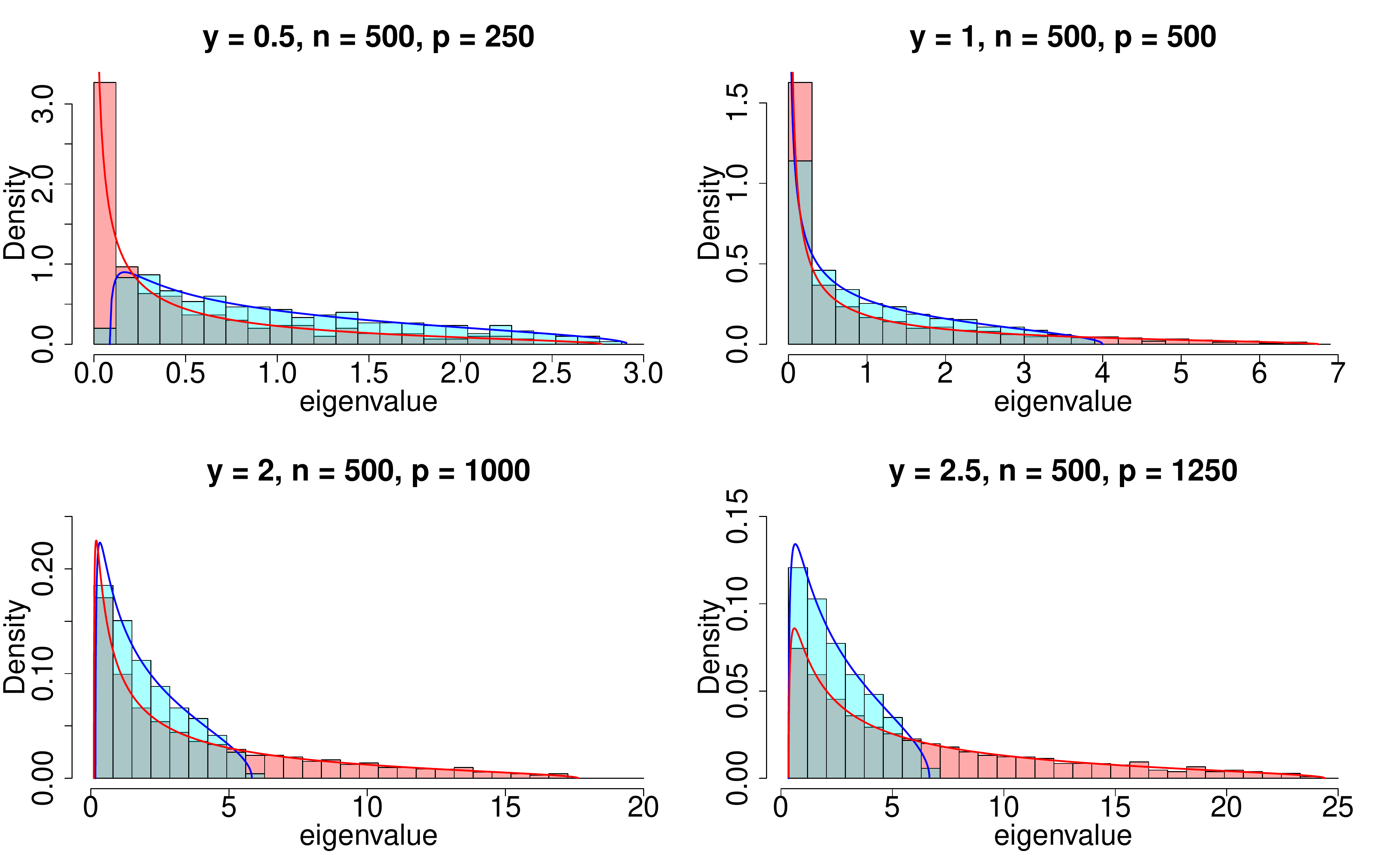}
    \caption{
    	 Histograms of the sample eigenvalues of the lag-1 sample auto-correlation matrix $\bR_{\tau}^{*}$ with $\tau = 1$ and the sample correlation matrix $\bR_{0}^{\bep}$ (light blue). Theoretical density functions of the LSDs of $\bR_{\tau}^{*}$  (red) and $\bR_{0}^{\bep}$ (blue) are exhibited in lines. In all panels, the sample size $n$ is fixed at $n=500$, and the ratio of the dimensionality to the sample size  is set as $y =0.5,1,1.5,2$ from top to bottom and left to right, respectively.}
    \label{fig:my_labelcor}
\end{figure}
\autoref{fig:my_labelcor} contrasts the LSD of $\bR_{\tau}^{*}$ (solid red curve) versus the LSD of sample correlation matrix $\bR_{0}^{\bep}$ (solid blue curve) , and the ESD of $\bR_{\tau}^{*}$ (light red histogram) with $\tau = 1$ versus the ESD of $\bR_{0}^{\bep}$ (light blue histogram)  based on i.i.d. samples from the standard normal distribution with $y = 0.5,1,2,2.5$ and $n = 500$. Clearly, the figure shows that the LSD (or ESD) of $\bR_{\tau}^{*}$ has different shapes with that of $\bR_{0}^{\bep}$ for all $(p,n)$ combinations.

\section{Proofs} 
In this section, we provide the proofs of \autoref{Theorem1} and \autoref{Theorem2}. Actually, our results rely on the results of the lag-$\tau$ sample auto-covariance matrix, which has been derived by \citep{li2015singular}. The strategy of our proof of the LSD is to show that the LSD of $\bR_{\tau}^{*}$ is the same as the LSD of $\bS_{\tau}^{*}$. Meanwhile, since the largest eigenvalue of $\bS_{\tau}^{*}$ has been studied by \citep{wang2016moment}, we show that the largest eigenvalues of $\bR_{\tau}^{*}$ and $\bS_{\tau}^{*}$ converge to the same limit.

\subsection{Substitution principle}
We first introduce the substitution principle. For the sample covariance matrix 
\begin{equation*}
	\bS_{0}^{\bep} = \frac{1}{N}\sum\limits_{i = 1}^{n}(\bep_{i}-\bar{\bep})(\bep_{i}-\bar{\bep})^{t}
\end{equation*}  
where $N = n-1$ is the adjusted sample size, \citep{zheng2015substitution} proposed the substitution principle, i.e., if we consider the non-centered sample covariance matrix
\begin{equation*}
\tilde{\bS}_{0}^{\bep} = \frac{1}{n}\sum\limits_{i = 1}^{n }\bep_{i}\bep_{i}^{t},
\end{equation*}  
with ${\rm E}(\bep_{i}) = {\bf 0}$, the asymptotic results for eigenvalues of $\bS_{0}^{\bep}$  partially hold for matrix $\tilde{\bS}_{0}^{\bep}$. Specifically, as for the first order result, $\bS_{0}^{\bep}$ and $\tilde{\bS}_{0}^{\bep}$ share the same LSD, i.e., the Mar{\v{c}}enko-Pastur distribution $F_{y}$ with index $y = \lim p/n$. As for the second order result, the central limit theorem (CLT) for linear spectral statistics (LSS), i.e., the linear functional of eigenvalues 
$
\sum\limits_{j = 1}^{p} g(\lambda_{j}),
$ 
of $\tilde{\bS}_{0}^{\bep}$ and $\bS_{0}^{\bep}$ are different. Instead, we need to replace the sample size $n$ of $\tilde{\bS}_{0}^{\bep}$ by the adjusted sample size $N = n-1$, then the CLT for LSS of $\bS_{0}^{\bep}(n)$ and $\tilde{\bS}_{0}^{\bep}(N)$ would be the same. We found that similar rules apply for sample auto-covariance and auto-correlation matrices. 

In this paper, we focus on the first order results, hence we do not need to apply the substitution principle. Denote 
\begin{equation*}
\tilde{\bR}_{\tau}^{\bep}= [{\rm diag}(\tilde{\bS}_{0}^{\bep})]^{-1/2}\tilde{\bS}_{\tau}^{\bep}[{\rm diag}(\tilde{\bS}_{0}^{\bep})]^{-1/2},\qquad \tilde{\bS}_{\tau}^{\bep}= \frac{1}{n}\sum_{i = 1}^{n}\bep_{i}\bep_{i+\tau}^{t}
\end{equation*}
and
\begin{equation*}
\tilde{\bR}_{\tau}^{*} = \tilde{\bR}_{\tau}^{\bep}(\tilde{\bR}_{\tau}^{\bep})^{t}, 	\qquad 
\tilde{\bS}_{\tau}^{*} = \tilde{\bS}_{\tau}^{\bep}(\tilde{\bS}_{\tau}^{\bep})^{t},
\end{equation*}
the following lemmas hold. 

\begin{lemma}
Under the assumptions in \autoref{Theorem1}, for fixed $\tau \geq 1$, as $p,n \rightarrow \infty$,  the empirical spectral distribution $F^{\tilde{\bS}_{\tau}^{*}}$ almost surely converges to the same LSD as $F^{\bS_{\tau}^{*}}$.
\end{lemma}

\begin{lemma}
Under the assumptions in \autoref{Theorem1},  for fixed $\tau \geq 1$, as $p,n \rightarrow \infty$, the largest eigenvalue of $\tilde{\bS}_{\tau}^{*}$, $\lambda_{\rm max}({\tilde{\bS}_{\tau}^{*}})$ almost surely converges to the same limit as that of $\lambda_{\rm max}(\bS_{\tau}^{*})$.
\end{lemma}
\begin{lemma}
Under the assumptions in \autoref{Theorem1},  for fixed $\tau \geq 1$, as $p,n \rightarrow \infty$, the empirical spectral distribution $F^{\tilde{\bR}_{\tau}^{*}}$ almost surely converges to  the same LSD as $F^{\bR_{\tau}^{*}}$, the distribution with a density function given by \eqref{den}.
\end{lemma}

\begin{lemma}
Under the assumptions in \autoref{Theorem1},  for fixed $\tau \geq 1$, as $p,n \rightarrow \infty$, the largest eigenvalue of $\tilde{\bR}_{\tau}^{*}$, $\lambda_{\rm max}({\tilde{\bR}_{\tau}^{*}})$ almost surely converges to the same limit as that of $\lambda_{\rm max}(\bR_{\tau}^{*})$.
\end{lemma}

Based on these asymptotic equivalence results, in the following proofs we only consider the non-centered lag-$\tau$ sample auto-covariance matrix $\tilde{\bS}^{\bep}_{\tau}$ and the non-centered lag-$\tau$ sample auto-correlation matrix $\tilde{\bR}^{\bep}_{\tau}$.  

\subsection{Proof of Theorem 2.1}
\begin{lemma}
Under the assumptions in \autoref{Theorem1}, let $L^{4}\left(\cdot,\cdot\right)$ be the Levy distance, for fixed $\tau \geq 1$,  as $p,n \rightarrow \infty$,  we have
\begin{equation*}
L^{4}\left(F^{\tilde{\bR}^{*}_{\tau}},F^{\tilde{\bS}^{*}_{\tau}}\right)\rightarrow 0 \quad a.s.
\end{equation*}
\label{lemma1}
\end{lemma}
\begin{proof}
First we consider the case $\tau =1$. Suppose $\bep_{j}^{0} = \left(\epsilon_{1,j},\ldots,\epsilon_{n,j}\right)^{t}$, $\bep_{j}^{1} = \left(\epsilon_{2,j},\ldots,\epsilon_{n+1,j}\right)^{t}$, then we can define the non-centered sample auto-correlation matrix $\tilde{\bR}_{1}^{\bep}$ and the non-centered sample auto-covariance matrix $\tilde{\bS}_{1}^{\bep}$ as follows: 
\begin{equation*}
\tilde{\bR}_{1}^{\bep} = \bX_{0}^{t}\bX_{1} \qand \tilde{\bS}_{1}^{\bep} = \frac{1}{n}\bE_{0}^{t}\bE_{1}
\end{equation*} 
where $\bX_{0} =\left(\frac{\bep_{1}^{0}}{\Vert \bep_{1}^{0}\Vert},\ldots,\frac{\bep_{p}^{0}}{\Vert \bep_{p}^{0}\Vert}\right)$, $\bX_{1} =\left(\frac{\bep_{1}^{1}}{\Vert \bep_{1}^{0}\Vert},\ldots,\frac{\bep_{p}^{1}}{\Vert \bep_{p}^{0}\Vert}\right)$, $\bE_{0} = \left(\bep_{1}^{0},\ldots,\bep_{p}^{0}\right)$ and $\bE_{1} = \left(\bep_{1}^{1},\ldots,\bep_{p}^{1}\right)$. 

By the difference inequality, we have
\begin{equation*}
L^{4}\left(F^{\tilde{\bR}_{1}^{*}},F^{\tilde{\bS}_{1}^{*}}\right) \leq \frac{2}{p}{\rm tr}\left(\left(\tilde{\bR}_{1}^{\bep}-\tilde{\bS}_{1}^{\bep}\right)\left(\tilde{\bR}_{1}^{\bep}-\tilde{\bS}_{1}^{\bep}\right)^{t}\right)\cdot \frac{1}{p}{\rm tr}\left(\tilde{\bR}_{1}^{*}+\tilde{\bS}_{1}^{*}\right) \coloneqq 2\cdot W_{1}\cdot W_{2}. 
\end{equation*} 
For $W_{2} = \frac{1}{p}{\rm tr}\left(\tilde{\bR}_{1}^{*}+\tilde{\bS}_{1}^{*}\right)$, we need to prove $W_{2} \rightarrow C_{1} \; a.s.$, where $C_{1}$ is a positive constant. Note that 
\begin{eqnarray*}
\frac{1}{p}{\rm tr}\left(\tilde{\bS}_{1}^{*}\right) &=& \frac{1}{p}\sum\limits_{j = 1}^{p}\sum\limits_{k = 1}^{p}\frac{{\bep_{j}^{0}}^{t}\bep_{k}^{1} {\bep_{k}^{1}}^{t}\bep_{j}^{0}}{n^{2}} \\
&=& \frac{1}{p}\sum\limits_{j = 1}^{p}\sum\limits_{k = 1}^{p}\frac{\left(\sum\limits_{i = 1}^{n}\epsilon_{i,j}\epsilon_{i+1,k}\right)^{2}}{n^{2}} \\
&=& \frac{1}{p}\sum\limits_{j = 1}^{p}\sum\limits_{k = 1}^{p}\frac{\sum\limits_{i = 1}^{n}\epsilon_{i,ij}^{2}\epsilon_{i+1,k}^{2}}{n^{2}} + \frac{1}{p}\sum\limits_{j = 1}^{p}\sum\limits_{k = 1}^{p}\frac{\sum\limits_{i_{1} \neq i_{2}}^{n}\epsilon_{i_{1},j}\epsilon_{i_{1}+1,k}\epsilon_{i_{2},j}\epsilon_{i_{2}+1,k}}{n^{2}}\\
&\coloneqq& Q_{1} + Q_{2}.
\end{eqnarray*}
For the term $Q_{1}$, if $p/n \rightarrow y >0$ we have 
\begin{equation*}
Q_{1} = \frac{p}{n}\cdot\frac{1}{np^{2}}\sum\limits_{j = 1}^{p}\sum\limits_{k = 1}^{p}\sum\limits_{i = 1}^{n}\epsilon_{i,j}^{2}\epsilon_{i+1,k}^{2} \rightarrow y \cdot 1 \quad a.s.
\end{equation*}
based on the law of large numbers. 

For the term $Q_{2}$,
\begin{eqnarray*}
{\rm E}\left(Q_{2}\right) &=& \frac{1}{pn^{2}}{\rm E}\sum\limits_{j = 1}^{p}\sum\limits_{k = 1}^{p}\sum\limits_{i_{1} \neq i_{2}}^{n}\epsilon_{i_{1},j}\epsilon_{i_{1}-1,k}\epsilon_{i_{2},j}\epsilon_{i_{2}-1,k} \\
&=& \frac{1}{pn^{2}}\sum\limits_{j = 1}^{p}\sum\limits_{k = 1}^{p}\sum\limits_{i_{1} \neq i_{2}}^{n}{\rm E}\left(\epsilon_{i_{1},j}\right){\rm E}\left(\epsilon_{i_{1}-1,k}\right){\rm E}\left(\epsilon_{i_{2},j}\right){\rm E}\left(\epsilon_{i_{2}-1,k}\right) \\
&=& 0.
\end{eqnarray*}
\begin{eqnarray*}
{\rm Var}\left(Q_{2}\right) &=& \frac{1}{p^{2}n^{4}}{\rm E}\left(\sum\limits_{j = 1}^{p}\sum\limits_{k = 1}^{p}\sum\limits_{i_{1} \neq i_{2}}^{n}\epsilon_{i_{1},j}\epsilon_{i_{1}+1,k}\epsilon_{i_{2},j}\epsilon_{i_{2}+1,k}\right)^{2} \\
&=& \frac{1}{p^{2}n^{4}}\sum\limits_{j = 1}^{p}\sum\limits_{k = 1}^{p}\sum\limits_{i_{1} \neq i_{2}}^{n}{\rm E}\left(\epsilon_{i_{1},j}^{2}\right){\rm E}\left(\epsilon_{i_{1}+1,k}^{2}\right){\rm E}\left(\epsilon_{i_{2},j}^{2}\right){\rm E}\left(\epsilon_{i_{2}+1,k}^{2}\right) \\
&=& O\left(\frac{1}{n^{2}}\right).
\end{eqnarray*}
According to Chebyshev's inequality, for any $\epsilon >0$
\begin{equation*}
P(|Q_{2}|>\epsilon) \leq \frac{{\rm Var}\left(Q_{2}\right)}{\epsilon^{2}} = O\left(\frac{1}{n^{2}}\right),
\end{equation*}
which is summable. Hence, based on Borel-cantelli lemma,
\begin{equation*}
Q_{2}\rightarrow 0 \quad a.s.
\end{equation*}
 we have
\begin{equation*}
\frac{1}{p}{\rm tr}\left(\tilde{\bS}_{1}^{*}\right) \rightarrow y \quad a.s.
\end{equation*}

For $\frac{1}{p}{\rm tr}\left(\tilde{\bR}_{1}^{*}\right)$, we obtain that
\begin{eqnarray*}
&&\left\vert\frac{1}{p}{\rm tr}\left(\tilde{\bR}_{1}^{*}\right) - \frac{1}{p}{\rm tr}\left(\tilde{\bS}_{1}^{*}\right)\right\vert \\
&=& \left\vert\frac{1}{p}\sum\limits_{j = 1}^{p}\sum\limits_{k = 1}^{p}\frac{{\bep_{j}^{0}}^{t}\bep_{k}^{1} {\bep_{k}^{1}}^{t}\bep_{j}^{0}}{\Vert \bep_{j}^{0}\Vert^{2} \Vert \bep_{k}^{0}\Vert^{2}}-\frac{1}{p}\sum\limits_{j = 1}^{p}\sum\limits_{k = 1}^{p}\frac{{\bep_{j}^{0}}^{t}\bep_{k}^{1} {\bep_{k}^{1}}^{t}\bep_{j}^{0}}{n^{2}}\right\vert \\
&\leq & \frac{1}{p}\sum\limits_{j = 1}^{p}\sum\limits_{k = 1}^{p}\frac{{\bep_{j}^{0}}^{t}\bep_{k}^{1} {\bep_{k}^{1}}^{t}\bep_{j}^{0}}{n^{2}}\left\vert\frac{n^{2}}{\Vert \bep_{j}^{0}\Vert^{2} \Vert \bep_{k}^{0}\Vert^{2}}-1\right\vert \\
&=& \frac{1}{p}\sum\limits_{j = 1}^{p}\sum\limits_{k = 1}^{p}\frac{{\bep_{j}^{0}}^{t}\bep_{k}^{1} {\bep_{k}^{1}}^{t}\bep_{j}^{0}}{n^{2}}\left\vert \left(\frac{n}{\Vert \bep_{j}^{0}\Vert^{2}}-1\right) \left(\frac{n}{\Vert \bep_{k}^{0}\Vert^{2}}-1\right)+\frac{n}{\Vert \bep_{j}^{0}\Vert^{2}}-1+ \frac{n}{\Vert \bep_{k}^{0}\Vert^{2}}-1 \right\vert \\
&\leq & \frac{1}{p}\sum\limits_{j = 1}^{p}\sum\limits_{k = 1}^{p}\frac{{\bep_{j}^{0}}^{t}\bep_{k}^{1} {\bep_{k}^{1}}^{t}\bep_{j}^{0}}{n^{2}}\left[\left\vert \left(\frac{n}{\Vert \bep_{j}^{0}\Vert^{2}}-1\right)\left(\frac{n}{\Vert \bep_{k}^{0}\Vert^{2}}-1\right)\right\vert+\left\vert\frac{n}{\Vert \bep_{j}^{0}\Vert^{2}}-1\right\vert+\left\vert\frac{n}{\Vert \bep_{k}^{0}\Vert^{2}}-1\right\vert
\right] \\
&\leq & \frac{1}{p}{\rm tr}\left(\tilde{\bS}_{1}^{*}\right)\cdot \left[\left(\max\limits_{1\leq j \leq p}\left\vert\frac{n}{\Vert \bep_{j}^{0}\Vert^{2}}-1\right\vert\right)^{2}+2\max\limits_{1\leq j \leq p}\left\vert\frac{n}{\Vert \bep_{j}^{0}\Vert^{2}}-1\right\vert\right].
\end{eqnarray*}
Since ${\rm E}|\epsilon_{1,1}|^{4} < \infty$, by the Lemma 2 from \citep{bai1993limit}, we know
\begin{equation*}
\max\limits_{1\leq j \leq p}\left\vert\frac{\sum_{i=1}^{n} \epsilon_{i,j}^{2}}{n}-1\right\vert \rightarrow 0 \quad a.s.,
\end{equation*}
and this implies that
\begin{equation*}
\max\limits_{1\leq j \leq p}\left\vert\frac{n}{\sum_{i=1}^{n} \epsilon_{i,j}^{2}}-1\right\vert \rightarrow 0 \quad a.s..
\end{equation*}
Since $\frac{1}{p}{\rm tr}\left(\tilde{\bS}_{1}^{*}\right)$ converges to a constant $y$ which has been shown above,  we have 
\begin{equation*}
\left\vert\frac{1}{p}{\rm tr}\left(\tilde{\bR}_{1}^{*}\right) - \frac{1}{p}{\rm tr}\left(\tilde{\bS}_{1}^{*}\right)\right\vert \rightarrow 0 \quad a.s.
\end{equation*}
It follows that $\frac{1}{p}{\rm tr}\left(\tilde{\bR}_{1}^{*}\right) \rightarrow y \;$ a.s., and then
\begin{equation*}
W_{2} = \frac{1}{p}{\rm tr}\left(\tilde{\bR}_{1}^{*}\right)+\frac{1}{p}{\rm tr}\left(\tilde{\bS}_{1}^{*}\right)\rightarrow 2y\quad a.s.
\end{equation*}
For the term of $W_{1}$, we have
\begin{eqnarray*}
W_{1} &=& \frac{1}{p}{\rm tr}\left(\left(\tilde{\bR}_{1}^{\bep}-\tilde{\bS}_{1}^{\bep}\right)\left(\tilde{\bR}_{1}^{\bep}-\tilde{\bS}_{1}^{\bep}\right)^{t}\right)\\
&=&  \frac{1}{p}\sum\limits_{j = 1}^{p}\sum\limits_{k = 1}^{p}\frac{{\bep_{j}^{0}}^{t}\bep_{k}^{1} {\bep_{k}^{1}}^{t}\bep_{j}^{0}}{n^{2}}- \frac{2}{p}\sum\limits_{j = 1}^{p}\sum\limits_{k = 1}^{p}\frac{{\bep_{j}^{0}}^{t}\bep_{k}^{1} {\bep_{k}^{1}}^{t}\bep_{j}^{0}}{n\cdot \Vert  \bep_{j}^{0}\Vert   \Vert \bep_{k}^{0}\Vert } + \frac{1}{p}\sum\limits_{j = 1}^{p}\sum\limits_{k = 1}^{p}\frac{{\bep_{j}^{0}}^{t}\bep_{k}^{1} {\bep_{k}^{1}}^{t}\bep_{j}^{0}}{\Vert \bep_{j}^{0}\Vert^{2} \Vert  \bep_{k}^{0}\Vert^{2}}\\
&\coloneqq & Q_{3} - 2Q_{4},
\end{eqnarray*}
where 
\begin{eqnarray*}
Q_{3} &=& \frac{1}{p}\sum\limits_{j = 1}^{p}\sum\limits_{k = 1}^{p}\frac{{\bep_{j}^{0}}^{t}\bep_{k}^{1} {\bep_{k}^{1}}^{t}\bep_{j}^{0}}{\Vert \bep_{j}^{0}\Vert^{2} \Vert \bep_{k}^{0}\Vert^{2}} - \frac{1}{p}\sum\limits_{j = 1}^{p}\sum\limits_{k = 1}^{p}\frac{{\bep_{j}^{0}}^{t}\bep_{k}^{1} {\bep_{k}^{1}}^{t}\bep_{j}^{0}}{n^{2}} \\
Q_{4} &=& \frac{1}{p}\sum\limits_{j = 1}^{p}\sum\limits_{k = 1}^{p}\frac{{\bep_{j}^{0}}^{t}\bep_{k}^{1} {\bep_{k}^{1}}^{t}\bep_{j}^{0}}{n\cdot \Vert\bep_{j}^{0}\Vert  \Vert\bep_{k}^{0}\Vert}- \frac{1}{p}\sum\limits_{j = 1}^{p}\sum\limits_{k = 1}^{p}\frac{{\bep_{j}^{0}}^{t}\bep_{k}^{1} {\bep_{k}^{1}}^{t}\bep_{j}^{0}}{n^{2}}.
\end{eqnarray*}
For $Q_{4}$,
\begin{eqnarray*}
&&\left\vert\frac{1}{p}\sum\limits_{j = 1}^{p}\sum\limits_{k = 1}^{p}\frac{{\bep_{j}^{0}}^{t}\bep_{k}^{1} {\bep_{k}^{1}}^{t}\bep_{j}^{0}}{n\cdot \Vert \bep_{j}^{0}\Vert  \Vert\bep_{k}^{0}\Vert}- \frac{1}{p}\sum\limits_{j = 1}^{p}\sum\limits_{k = 1}^{p}\frac{{\bep_{j}^{0}}^{t}\bep_{k}^{1} {\bep_{k}^{1}}^{t}\bep_{j}^{0}}{n^{2}}\right\vert \\
&\leq & \frac{1}{p}\sum\limits_{j = 1}^{p}\sum\limits_{k = 1}^{p}\frac{{\bep_{j}^{0}}^{t}\bep_{k}^{1} {\bep_{k}^{1}}^{t}\bep_{j}^{0}}{n^{2}}\left\vert\frac{n}{\Vert \bep_{j}^{0}\Vert \Vert\bep_{k}^{0}\Vert}-1\right\vert \\
&=& \frac{1}{p}\sum\limits_{j = 1}^{p}\sum\limits_{k = 1}^{p}\frac{{\bep_{j}^{0}}^{t}\bep_{k}^{1} {\bep_{k}^{1}}^{t}\bep_{j}^{0}}{n^{2}}\left\vert \left(\frac{\sqrt{n}}{\sqrt{\Vert \bep_{j}^{0}\Vert^{2}}}-1\right)\left(\frac{\sqrt{n}}{\sqrt{\Vert \bep_{k}^{0}\Vert^{2}}}-1\right)+\frac{\sqrt{n}}{\sqrt{\Vert \bep_{j}^{0}\Vert^{2}}}-1+ \frac{\sqrt{n}}{\sqrt{\Vert \bep_{k}^{0}\Vert^{2}}}-1 \right\vert \\
&\leq & \frac{1}{p}{\rm tr}\left(\tilde{\bS}_{1}^{*}\right)\cdot \left[\left(\max\limits_{1\leq j \leq p}\left\vert\frac{\sqrt{n}}{\sqrt{\Vert \bep_{j}^{0}\Vert^{2}}}-1\right\vert\right)^{2}+2\max\limits_{1\leq j \leq p}\left\vert\frac{\sqrt{n}}{\sqrt{\Vert \bep_{j}^{0}\Vert^{2}}}-1\right\vert\right].
\end{eqnarray*}
According to Lemma 2 of \citep{bai1993limit}, we have
\begin{equation*}
\max\limits_{1\leq j \leq p}\left\vert\frac{\sqrt{n}}{\sqrt{\sum_{i=1}^{n} \epsilon_{i,j}^{2}}}-1\right\vert \rightarrow 0 \quad a.s.
\end{equation*}
Therefore $Q_{4}\rightarrow 0 \;$ a.s. Given that the following result
\begin{equation*}
Q_{3} = \frac{1}{p}{\rm tr}\left(\tilde{\bR}_{1}^{*}\right) - \frac{1}{p}{\rm tr}\left(\tilde{\bS}_{1}^{*}\right) \rightarrow 0 \quad a.s.
\end{equation*}
has been proved, we have 
\begin{equation*}
W_{1} = \frac{1}{p}{\rm tr}\left(\left(\tilde{\bR}_{1}^{\bep}-\tilde{\bS}_{1}^{\bep}\right)(\left(\tilde{\bR}_{1}^{\bep}-\tilde{\bS}_{1}^{\bep}\right)^{t}\right) \rightarrow 0\quad a.s.
\end{equation*}
Together with $W_{1},W_{2}$, we obtain
\begin{equation*}
L^{4}\left(F^{\tilde{\bR}_{1}^{*}},F^{\tilde{\bS}_{1}^{*}}\right) \rightarrow 0 \quad a.s.
\end{equation*} 
The procedure of the proof will not change for any given positive integer $\tau$. Therefore, we have
\begin{equation*}
 L^{4}\left(F^{\tilde{\bR}_{\tau}^{*}},F^{\tilde{\bS}_{\tau}^{*}}\right) \rightarrow 0 \quad a.s.
\end{equation*}
\end{proof}
\subsection{Proof of Theorem 2.2}
\begin{lemma}
Under the assumptions in \autoref{Theorem1}, a let $\lambda_{\rm max}({\tilde{\bS}_{\tau}^{*}})$ and $\lambda_{\rm max}({\tilde{\bR}_{\tau}^{*}})$ be the largest eigenvalues of $\tilde{\bS}_{\tau}^{*}$ and $\tilde{\bR}_{\tau}^{*}$, respectively. As $p,n \rightarrow \infty$,  we have
\begin{equation*}
\vert\sqrt{\lambda_{\rm max}(\tilde{\bR}_{\tau}^{*})}-\sqrt{\lambda_{\rm max}(\tilde{\bS}_{\tau}^{*})}\vert\rightarrow 0 \quad a.s. 
\end{equation*}
\end{lemma}
\begin{proof}
Denote $\bep_{j}^{0} = \left(\epsilon_{1,j},\ldots,\epsilon_{n,j}\right)^{t}$, $\bep_{j}^{\tau} = \left(\epsilon_{\tau,j},\ldots,\epsilon_{n+\tau,j}\right)^{t}$. Rewrite
\begin{equation*}
\tilde{\bR}_{\tau}^{\bep} = \frac{1}{n}\bD\bE_{0}^{t}\bE_{\tau}\bD \qand \tilde{\bS}_{\tau}^{\bep} = \frac{1}{n}\bE_{0}^{t}\bE_{\tau}
\end{equation*}
where $\bD = {\rm diag}\left(\frac{\sqrt{n}}{\Vert \bep_{1}^{0}\Vert},\ldots,\frac{\sqrt{n}}{\Vert \bep_{p}^{0}\Vert}\right)$, $\bE_{0} = \left(\bep_{1}^{0},\ldots,\bep_{p}^{0}\right)$ and $\bE_{\tau} = \left(\bep_{1}^{\tau},\ldots,\bep_{p}^{\tau}\right)$. 

Under the conditions of \autoref{Theorem1}, according to Theorem 4.1 from \citep{wang2016moment}, we have
\begin{equation}\label{111}
\lambda_{{\rm max}}(\tilde{\bS}_{\tau}^{*})\rightarrow  b \quad a.s.,
\end{equation}
where $b = \frac{1}{8}\left(-1+20y+8y^{2}+\left(1+8y\right)^{3/2}\right)$ is the right end point of the support of the LSD of  $\bR_{\tau}^{*}$. Our target is to show that 
\begin{equation}\label{eee}
\vert\sqrt{\lambda_{\rm max}(\tilde{\bR}_{\tau}^{*})}-\sqrt{\lambda_{\rm max}(\tilde{\bS}_{\tau}^{*})}\vert\rightarrow 0 \quad a.s. 
\end{equation}
For any matrix $\bA$, we denote $\Vert \bA\Vert_{2}$ as the spectrum norm of $\bA$, which is defined as the square root of the largest eigenvalue of $\bA\bA^{t}$. By Corollary 7.3.8 from \citep{horn1985matrix}, we have
\begin{equation*}
\vert\sqrt{\lambda_{\rm max}(\tilde{\bR}_{\tau}^{*})}-\sqrt{\lambda_{\rm max}(\tilde{\bS}_{\tau}^{*})}\vert \leq  \Vert \tilde{\bR}_{\tau} - \tilde{\bS}_{\tau}\Vert_{2}.
\end{equation*}
Meanwhile the spectrum norm satisfies the triangle inequality and $\Vert\bA\bC\Vert_{2}\leq \Vert \bA\Vert_{2}\cdot\Vert \bC\Vert_{2}$ for any $\bA$ and $\bC$, then we have
\begin{eqnarray}\label{1}
&& \Vert \tilde{\bR}_{\tau} - \tilde{\bS}_{\tau}\Vert_{2} \notag \\
&=& \Vert \frac{1}{n}\bD\bE_{0}^{t}\bE_{\tau}\bD - \frac{1}{n}\bE_{0}^{t}\bE_{\tau}\Vert_{2} \notag  \\
&=& \Vert \frac{1}{n}\bD\bE_{0}^{t}\bE_{\tau}\bD - \frac{1}{n}\bD\bE_{0}^{t}\bE_{\tau} +\frac{1}{n}\bD\bE_{0}^{t}\bE_{\tau} - \frac{1}{n}\bE_{0}^{t}\bE_{\tau}\Vert_{2} \notag \\
&\leq & \Vert \frac{1}{n}\bD\bE_{0}^{t}\bE_{\tau}\bD - \frac{1}{n}\bD\bE_{0}^{t}\bE_{\tau}\Vert_{2} +\Vert \frac{1}{n}\bD\bE_{0}^{t}\bE_{\tau} - \frac{1}{n}\bE_{0}^{t}\bE_{\tau}\Vert_{2}\notag  \\
& = & \Vert \frac{1}{n}\left(\bD-\bI+\bI\right)\bE_{0}^{t}\bE_{\tau}\left(\bD-\bI\right)\Vert_{2} +\Vert \frac{1}{n}\left(\bD-\bI\right)\bE_{0}^{t}\bE_{\tau}\Vert_{2}\notag  \\
& \leq & \Vert \frac{1}{n}\left(\bD-\bI\right)\bE_{0}^{t}\bE_{\tau}\left(\bD-\bI\right)\Vert_{2} +2\Vert \frac{1}{n}\left(\bD-\bI\right)\bE_{0}^{t}\bE_{\tau}\Vert_{2}\notag  \\
&\leq & \Vert \frac{1}{n}\bE_{0}^{t}\bE_{\tau}\Vert_{2} \cdot\Vert \bD-\bI \Vert_{2}^{2} + 2\Vert \frac{1}{n}\bE_{0}^{t}\bE_{\tau}\Vert_{2} \cdot\Vert \bD-\bI \Vert_{2}. 
\end{eqnarray}
Since ${\rm E}|\epsilon_{1,1}|^{4} < \infty$, by Lemma 2 of \citep{bai1993limit}, we know that
\begin{equation*}
\max\limits_{1\leq j \leq p}\left\vert\frac{\Vert \bep_{j}^{0}\Vert^{2}}{n}-1\right\vert \rightarrow 0 \quad a.s.,
\end{equation*}
which implies 
\begin{equation*}
\Vert \bD-\bI \Vert_{2} = \max\limits_{1\leq j \leq p}\left\vert \frac{\sqrt{n}}{\Vert \bep_{j}^{0}\Vert} - 1\right\vert\rightarrow 0\quad a.s.
\end{equation*}
This together with \autoref{111} and \autoref{1} proves \autoref{eee}.
\end{proof}

\bibliographystyle{imsart-nameyear}
\bibliography{Ref}
\end{document}